\documentclass[10pt,reqno, english]{amsart}  
\usepackage[utf8]{inputenc}
\usepackage[T1]{fontenc}
\usepackage{amsmath,amsthm}
\usepackage{amsfonts,amssymb,enumerate}
\usepackage{url,paralist}
\usepackage{mathtools}  
\usepackage[colorlinks=true,urlcolor=blue,linkcolor=red,citecolor=magenta]{hyperref}
\usepackage{enumerate}
\usepackage{anysize}

\theoremstyle{plain}
\newtheorem{theorem}{Theorem}[section]
\newtheorem{lemma}[theorem]{Lemma}

\newtheorem{corollary}[theorem]{Corollary}

\newtheorem{conjecture}[theorem]{Conjecture}
\newtheorem*{theorem*}{Theorem}

\newtheorem*{claim*}{Claim}

\theoremstyle{definition}

\newtheorem{problem}[theorem]{Problem}

\newcommand{\R}{\mathbb{R}}
\newcommand{\Z}{\mathbb{Z}}
\newcommand{\Sym}{\mathfrak{S}}
\newcommand{\dist}{\mathrm{dist}}
\newcommand{\pt}{\mathrm{pt}}

\DeclareMathOperator{\conv}{\mathrm{conv}}

\begin{document}

\title[Affine Tverberg results without continuous generalization]{On affine Tverberg-type results without\\ continuous generalization}



\author{Florian Frick}
\address{Dept.\ Math., Cornell University, Ithaca, NY 14853, USA}
\email{ff238@cornell.edu}

\date{February 17, 2017}
\maketitle


\begin{abstract}
\small
Recent progress building on the groundbreaking work of Mabillard and Wagner has shown
that there are important differences between the affine and continuous theory for Tverberg-type results.
These results aim to describe the intersection pattern of convex hulls of point sets in Euclidean space and continuous relaxations thereof.
Here we give additional examples of an affine-continuous divide, but our deductions are almost elementary and
do not build on the technical work of Mabillard and Wagner. Moreover, these examples show a difference between the
affine and continuous theory even asymptotically for arbitrarily large complexes.
Along the way we settle the Tverberg 
admissible-prescribable problem (or AP conjecture) in the negative, give a new, short and elementary proof of the
balanced case of the AP conjecture which was recently proven by Joji\'c, Vre\'cica, and \v Zivaljevi\'c in a series of
two papers, provide examples of Tverberg-type results
that hold affinely but not continuously without divisibility conditions on the intersection multiplicity,
extend a result of Sober\'on, and show that this extension has a topological generalization if and
only if the intersection multiplicity is a prime. 
\end{abstract}

\section{Introduction}

A simplicial complex $K$ that admits a continuous embedding into~$\R^d$ is not necessarily embeddable into~$\R^d$ in a 
facewise affine way. Brehm~\cite{brehm1983} found a triangulation of the
M\"obius strip that cannot be affinely embedded into~$\R^3$.  Bokowski and Guedes de Oliveira~\cite{bokowski2000} showed
that there is a triangulation of the orientable surface of genus six that does not admit a facewise affine embedding into~$\R^3$.
Later Schewe~\cite{schewe2010} found a triangulation of the orientable genus-five surface that cannot be embedded into~$\R^3$
with affine faces.

Tverberg-type theory is a natural generalization of the theory of embeddings of simplicial complexes, where now one is
more generally interested in the intersection pattern of images of pairwise disjoint faces in a simplicial complex when mapped,
either affinely or continuously, to Euclidean space. In addition to $2$-fold intersections, as in the case of embeddings, one
also strives to understand $r$-fold intersections among pairwise disjoint faces. In this generalized setting one encounters
surprising phenomena: 
\begin{compactenum}[1.]
	\item we have to distinguish between affine and continuous maps even for simplices;
	\item and this distinction between the affine and continuous theory depends on divisibility properties of the intersection multiplicity.
\end{compactenum}
Tverberg~\cite{tverberg1966} showed that any affine map $\Delta_{(r-1)(d+1)} \longrightarrow \R^d$ from
a simplex of dimension~$(r-1)(d+1)$ to $\R^d$ has a point of $r$-fold coincidence among its pairwise disjoint faces.
The faces involved in this intersection form a \emph{Tverberg partition} for~$f$.
Tverberg's theorem remains true for continuous maps if $r$ is a power of a prime, see B\'ar\'any, Shlosman, and Sz\H ucs~\cite{barany1981}, and \"Ozaydin~\cite{ozaydin1987}, but
is false if $r$ has at least two distinct prime divisors, as was shown in~\cite{blagojevic2015, frick2015} building on work of Mabillard and Wagner~\cite{mabillard2015}.
Also see the recent survey by B\'ar\'any, Blagojevi\'c, and Ziegler~\cite{barany2016}.

Here we show -- for the first time -- that a qualitative difference between the affine and continuous theory for Tverberg-type 
results persists asymptotically, for arbitrarily large complexes and fixed dimension. Our deductions do not depend on the technical work of 
Mabillard and Wagner, but are significantly more elementary. We exhibit an affine-continuous dichotomy for any intersection
multiplicity $r \ge 3$, independent of divisibility properties of~$r$, and our examples are low-dimensional: we construct maps
to~$\R^3$.

The Tverberg admissible-prescribable problem aims to characterize which dimensions must occur for Tverberg partitions
for any continuous map $f\colon \Delta \longrightarrow \R^d$ from a sufficiently large simplex~$\Delta$. The AP conjecture,
popularized by \v Zivaljevi\'c in the most recent edition of the \emph{Handbook of Discrete and Computational Geometry}~\cite[Section 21.4.3]{zivaljevic2017},
asserts that for $r$ a prime power these dimensions are given by the two conditions  $\sum d_i \ge (r-1)d$ and $d \ge d_i \ge \lfloor \frac{d}{2} \rfloor$.
Outside of the prime power case the constructions of counterexamples to the topological Tverberg conjecture already imply that
the AP conjecture must fail as well.
One consequence of the constructions of the present manuscript is that the AP conjecture fails for all intersection multiplicities~${r \ge 3}$; 
see Corollary~\ref{cor:prescr}:

\begin{theorem*}
	The AP conjecture fails for any~$r \ge 3$, already in dimension $d = 3$.
\end{theorem*}

The AP conjecture holds for $r=2$; see~\cite[Thm.~6.8]{blagojevic2015}. If we require the map $f\colon \Delta \longrightarrow \R^d$
above to be affine, we get an affine version of the AP conjecture, where we do not need to require that $r$ be a power of a prime.
Unlike its continuous relaxation we can show that this affine AP conjecture is true in~$\R^3$; see Theorem~\ref{thm:3-dim-prescr}:

\begin{theorem*}
	The affine AP conjecture holds in dimension $d \le 3$.
\end{theorem*}

Here we collect additional results of this manuscript and where to find them:
\begin{compactitem}
	\item We give a new short and elementary proof of the balanced case of the AP conjecture, which was recently 
	proven in a sequence of two papers by Joji\'c, Vre\'cica, and \v Zivaljevi\'c~\cite{jojic2017, jojic2015}; see Theorem~\ref{thm:balanced}.
	This solves a problem raised in~\cite{jojic2017} on whether the balanced AP Conjecture can be proven by ``indirect methods''.
	\item We show that the dimensions that must occur for Tverberg partitions for any continuous map $f\colon \Delta \longrightarrow \R^3$
	from a sufficiently large simplex~$\Delta$ are different from those dimensions that must occur for affine maps; see Corollary~\ref{cor:prescr}.
	A class of examples of dimensions that must occur for affine maps was recently found by Bukh, Loh, and Nivasch~\cite{bukh2016}.
	\item In fact, Bukh, Loh, and Nivasch prove special cases of a conjecture (Conjecture~\ref{conj}) that gives an asymptotic characterization of 
	all affine Tverberg partitions based on order types. P\'or~\cite{por2016} announced a proof of the full conjecture, and this
	implies the affine version of the AP conjecture in full generality; see Theorem~\ref{thm:conj-implies-prescr}. 
	Thus the topological analog of Conjecture~\ref{conj} must fail; see Corollary~\ref{cor:prescr}.
	\item We use our geometric constructions to show that certain symmetric multiple chessboard complexes are not shellable; see Theorem~\ref{thm:nonshellable}.
	In particular, geometric constructions of maps to~$\R^d$ can be used as obstructions to the shellability of symmetric simplicial complexes.
	\item We extend a result of Sober\'on~\cite{soberon2015} on Tverberg partitions where the points of coincidence have equal barycentric coordinates
	to an orbit-collapsing result; see Theorem~\ref{thm:barycentric}. This solves a problem raised by Sober\'on~\cite[Section~4]{soberon2015}.
	\item We remark that this orbit-collapsing result has a topological analog if the intersection multiplicity $r$ is a prime -- see Theorem~\ref{thm:top-orbit} -- and
	fails otherwise; see Theorem~\ref{thm:top-orbit-fails}.
\end{compactitem}

\noindent
\textbf{Acknowledgements.}
I would like to thank Boris Bukh, Albert Haase, Attila P\'or, and Uli Wagner for helpful discussions.

\section{The Tverberg admissible-prescribable problem}

Given $2d+3$ points in~$\R^{2d}$ there are two disjoint subsets of at most $d+1$ points each such that their
convex hulls intersect. The original proofs of this result due to Van Kampen~\cite{vanKampen1933} and independently Flores~\cite{flores1933} are topological and lead to the following
topological generalization: let $f\colon \Delta_{2d+2} \longrightarrow \R^{2d}$ be continuous; then there are
two disjoint $d$-dimensional faces $\sigma_1$ and $\sigma_2$ of~$\Delta_{2d+2}$ such that $f(\sigma_1) \cap f(\sigma_2) \ne \emptyset$. 
The case $d=1$ gives that the complete graph $K_5$ is nonplanar, and is actually equivalent to it because of the Hanani--Tutte theorem.
This higher-dimensional analog of the nonplanarity of~$K_5$ has a generalization to $r$-fold intersections.
Perhaps surprisingly, this generalization holds if and only if $r$ is a power of a prime. More precisely, Sarkaria~\cite{sarkaria1991} for primes $r$
and Volovikov~\cite{volovikov1996} in the general case that $r$ is a power of a prime proved that for any continuous map $f\colon \Delta_{(r-1)(rd+2)} \longrightarrow \R^{rd}$
there are $r$ pairwise disjoint $(r-1)d$-dimensional faces $\sigma_1, \dots, \sigma_r$ of $\Delta_{(r-1)(rd+2)}$ such that
$f(\sigma_1) \cap \dots \cap f(\sigma_r) \ne \emptyset$. That this result fails for any $r$ with at least two distinct prime divisors
follows by combining the work of Mabillard and Wagner~\cite{mabillard2015} with \"Ozaydin's~\cite{ozaydin1987}; see~\cite{blagojevic2015}. 
The affine version of this result, that is, the statement that for any $(r-1)(rd+2)+1$ points $X \subseteq \R^{rd}$ there are pairwise disjoint subsets 
$X_1, \dots, X_r \subseteq X$ with $|X_i| = (r-1)d+1$ such that $\conv X_1 \cap \dots \cap \conv X_r \ne \emptyset$, follows from the topological version for $r$ a power of
a prime but is open for other intersection multiplicities~$r$.

A natural question is which dimensions can be prescribed for the intersecting faces.
Let $r\ge 2$ and $d \ge 1$ be integers and let $0 \le d_1 \le d_2 \le \dots \le d_r$ be a non-decreasing $r$-tuple 
of integers. We call $(d_1, \dots, d_r)$ \emph{Tverberg prescribable} for $r$ and $d$ if there is an $N$ such that for
every continuous map $f\colon \Delta_N \longrightarrow \R^d$ there are $r$ pairwise
disjoint faces $\sigma_1, \dots, \sigma_r$ of~$\Delta_N$ with $f(\sigma_1) \cap \dots \cap f(\sigma_r) \ne \emptyset$
and $\dim \sigma_i = d_i$ for all~$i$. We call $(d_1, \dots, d_r)$ \emph{affinely Tverberg prescribable} for $r$ and $d$
if there is an $N$ such that for every affine map $f\colon \Delta_N \longrightarrow \R^d$ there are $r$ pairwise
disjoint faces $\sigma_1, \dots, \sigma_r$ of~$\Delta_N$ with $f(\sigma_1) \cap \dots \cap f(\sigma_r) \ne \emptyset$
and $\dim \sigma_i = d_i$ for all~$i$. Equivalently, every sufficiently large point set in~$\R^d$ has $r$ pairwise disjoint
subsets $X_1, \dots, X_r$ with $|X_i| = d_i +1$ and such that $\conv X_1 \cap \dots \cap \conv X_r \ne \emptyset$.

Clearly every Tverberg prescribable sequence is affinely Tverberg prescribable as well. Here we will show that the 
converse does not hold: for every $r \ge 3$ there is a sequence of nonnegative integers that is affinely Tverberg
prescribable but we cannot force faces of these dimensions to intersect for general continuous maps.

For $(d_1, \dots, d_r)$ to be Tverberg prescribable, the $d_i$ have to be large. This can be seen by placing an
arbitrarily large point set on the moment curve $\gamma(t) = (t,t^2, \dots, t^d)$ in~$\R^d$. Then every set of size 
at most $\lfloor\frac{d}{2}\rfloor$ has its convex hull on
the boundary of the convex hull of the entire point set (which is a cyclic polytope) and thus does not intersect
the convex hull of all other points. In particular, there is no $r$-fold Tverberg point. We conclude that in order to be
Tverberg prescribable the $d_i$ must be at least~$\lfloor\frac{d}{2}\rfloor$. We call an $r$-tuple of integers $(d_1, \dots, d_r)$
with $\sum_i d_i = (r-1)d$ and $\lfloor\frac{d}{2}\rfloor \le d_1 \le d_2 \le \dots \le d_r \le d$ \emph{Tverberg admissible}.
The Tverberg admissible-prescribable problem asks whether every Tverberg admissible $r$-tuple is also
Tverberg prescribable. This would completely characterize the dimensions of faces that must occur asymptotically
in Tverberg partitions.

\begin{conjecture}[AP conjecture~{\cite{blagojevic2014}, \cite[Problem 21.4.12]{zivaljevic2017}}]
\label{conj:AP}
	Every Tverberg admissible $r$-tuple is Tverberg prescribable for $r$ a power of a prime:
	if the $r$ integers $d_i$ with $\lfloor\frac{d}{2}\rfloor \le d_i \le d$ satisfy $\sum_i d_i = (r-1)d$ then there
	is a sufficiently large simplex~$\Delta$ such that any continuous map $f \colon \Delta \longrightarrow \R^d$
	identifies $r$ points from $r$ pairwise disjoint faces~$\sigma_1, \dots, \sigma_r$ of~$\Delta$ with
	$\dim \sigma_i = d_i$.
\end{conjecture}

In~\cite{zivaljevic2017} and in~\cite{blagojevic2014} this problem is stated for any~$r$, but the counterexamples to the topological
Tverberg conjecture show the need to add the condition that $r$ be a power of a prime. 
The Van Kampen--Flores theorem asserts that the AP conjecture holds if $r = 2$ and the dimension $d$ is even. The same is true for the
odd-dimensional case; see~\cite[Thm.~6.8]{blagojevic2015}. Joji\'c, Vre\'cica, and \v Zivaljevi\'c describe the AP conjecture
as a ``very interesting problem''; see~\cite[Section 1.3]{jojic2015}.
We will construct counterexamples to the AP conjecture for any $r \ge 3$.
We can formulate an affine version of Conjecture~\ref{conj:AP} for any~$r$:

\begin{conjecture}[affine AP conjecture]
\label{conj:affine-AP}
	Every Tverberg admissible $r$-tuple is affinely Tverberg prescribable:
	if the $r$ integers $d_i$ with $\lfloor\frac{d}{2}\rfloor \le d_i \le d$ satisfy $\sum_i d_i = (r-1)d$ then in any 
	sufficiently large finite point set $X \subseteq \R^d$ there are pairwise disjoint subsets $X_1, \dots, X_r \subseteq X$
	such that $\conv X_1 \cap \dots \cap \conv X_r \ne \emptyset$ and $|X_i| = d_i +1$.
\end{conjecture}

The AP conjecture and its affine relative have recently received the attention of several authors:
Before the conjecture was formulated Sarkaria~\cite{sarkaria1991} and Volovikov~\cite{volovikov1996} showed the AP Conjecture holds for Tverberg admissible
$r$-tuples $(d_1, \dots, d_r)$ with $d_1 = d_2 = \dots = d_r$. Slight extensions of this are contained in~\cite{blagojevic2014}.
Joji\'c, Vre\'cica, and \v Zivaljevi\'c~\cite{jojic2017, jojic2015} proved the balanced case of the AP conjecture in a series of two technical papers; here a Tverberg 
admissible $r$-tuple $(d_1, \dots, d_r)$ is \emph{balanced} if $|d_i-d_j| \le 1$ for all $i$ and~$j$. We will give a new and simple proof of 
the balanced case of the AP Conjecture; see Theorem~\ref{thm:balanced}. The affine AP conjecture was implicitly treated by Bukh, Loh,
and Nivasch~\cite{bukh2016}. 
As a corollary we will obtain here that Conjecture~\ref{conj:affine-AP} holds in dimensions $d \le 3$; see Theorem~\ref{thm:3-dim-prescr}.
A result recently announced
by P\'or would imply the affine AP conjecture in full generality; see Theorem~\ref{thm:conj-implies-prescr}.

We found the conditions that make an $r$-tuple $(d_1, \dots, d_r)$ Tverberg admissible by placing points on the moment curve.
In order to extend this argument we need a complete understanding of $r$-fold intersections of convex hulls of 
point sets on the moment curve. The intersection combinatorics of point sets on the stretched moment curve are completely
understood; see Bukh, Loh, and Nivasch~\cite{bukh2016} (according to the authors of~\cite{bukh2016} this was independently 
observed by B\'ar\'any and P\'or as well as Mabillard and Wagner): the Tverberg partitions that occur on the stretched moment
curve are precisely the colorful ones. Here a partition of $\{1,2, \dots, (r-1)(d+1)+1\}$ into $r$ parts $X_1, \dots, X_r$ is called
\emph{colorful} if for each $1 \le k \le d+1$ the set $Y_k = \{(r-1)(k-1)+1, \dots, (r-1)k+1\}$ satisfies $|Y_k \cap X_i| = 1$ for all~$i$.
We say that a partition $X_1 \sqcup \dots \sqcup X_r$ of $\{1,2, \dots, (r-1)(d+1)+1\}$ \emph{occurs} as a Tverberg partition in
a sequence $x_1, \dots, x_N$ of points in~$\R^d$ if there is a subsequence $x_{i_1}, \dots, x_{i_n}$ of length $n = (r-1)(d+1)+1$
such that the sets $\conv\{x_{i_k} \: | \: k \in X_j\}$ all share a common point. It has been conjectured that asymptotically
only colorful Tverberg partitions must occur.

\begin{conjecture}[Bukh, Loh, and Nivasch~{\cite[Conjecture 1.3]{bukh2016}}]
\label{conj}
	Any colorful partition of the index set $\{1,2, \dots, {(r-1)(d+1)+1}\}$ into $r$ parts occurs as a Tverberg partition in any
	sufficiently long sequence of generic points in~$\R^d$. For all other partitions there are arbitrarily large point sets
	where the partition does not occur as a Tverberg partition.
\end{conjecture}

We will show that a topological generalization of this conjecture fails; see Corollary~\ref{cor:prescr}.
P\'or~\cite{por2016} recently announced a proof of this conjecture. Here a point set is generic if it is not in the zero set
of a finite family of polynomials that express certain geometric predicates; see~\cite{bukh2016} for details. It will only be important for us that these generic
point sets are dense in the space of all point sets. Thus by a standard limiting argument it suffices to prove results for 
these generic point sets. 
Bukh, Loh, and Nivasch settle certain special cases of Conjecture~\ref{conj}:

\begin{theorem}[Bukh, Loh, and Nivasch~{\cite[Theorem 1.4]{bukh2016}}]
\label{thm:dim-3}
	Conjecture~\ref{conj} holds for $d \le 2$ and all~$r$. It also holds for $d = 3$ for all partitions that have sizes $2,3,4,\dots, 4$.
\end{theorem}

Some additional cases are treated in~\cite{bukh2016}. Here we will use Theorem~\ref{thm:dim-3} combined with recent work
of Joji\'c, Vre\'cica, and \v Zivaljevi\'c~\cite{jojic2015} to prove the affine AP conjecture in dimensions~${d \le 3}$. 
The AP conjecture obviously holds in the real line for~${d = 1}$. The only Tverberg admissible $r$-tuple for ${d = 2}$ is 
$(1,1,2,\dots, 2)$.
Characterizing all Tverberg admissible $r$-tuples for $d = 3$ is simple: if integers $d_1, \dots, d_r$ with $1 \le d_i \le 3$ satisfy $\sum_i d_i = 3(r-1)$,
then the $r$-tuple $(d_1, \dots, d_r)$ either is $(1,2,3,\dots,3)$ or $(2,2,2,3,\dots,3)$. It is not too difficult to see from~\cite{bukh2016}
and~\cite{jojic2015} that in any sufficiently large point set there are pairwise disjoint subsets of respective cardinalities $d_i+1$
whose convex hulls share a common point:

\begin{theorem}
\label{thm:3-dim-prescr}
	The affine AP conjecture holds in dimensions $d \le 3$.
\end{theorem}

\begin{proof}
	In~\cite[Section 3.6]{bukh2016} colorful $r$-partitions of $\{1, 2, \dots, (r-1)(3+1)+1\}$ with parts of sizes $2,3,4,\dots, 4$
	are constructed. Theorem~\ref{thm:dim-3} then shows that in any sufficiently large 
	generic point set in~$\R^3$ we can find $r$ pairwise disjoint sets of sizes $2,3,4,\dots, 4$ whose convex hulls share a common point.
	A limiting argument extends this to nongeneric point sets. 
	For this limiting argument notice that in a generic point set the point of $r$-fold intersection lies in the relative interior of the 
	involved faces for codimension reasons. This intersection can thus not disappear in the limit, while the point of intersection might 
	lie in a face of lower dimension. This is not a problem as we can always add points to such a face, while keeping all faces vertex-disjoint.
	Thus the $r$-tuple $(1,2,3, \dots, 3)$ is affinely Tverberg prescribable.
	The same reasoning shows that $(1,1,2,\dots, 2)$ is affinely Tverberg prescribable in dimension~${d = 2}$.
	The $r$-tuple $(2,2,2,3,\dots,3)$ is balanced and thus Tverberg prescribable provided that $r$ is a prime power, even for general 
	continuous maps, by a result of Joji\'c, Vre\'cica, and \v Zivaljevi\'c~\cite{jojic2015}. We give a significantly simplified proof of that result;
	see Theorem~\ref{thm:balanced}. Note that it is sufficient to prove that the $r$-tuple $(2,2,2,3,\dots,3)$ is Tverberg prescribable
	for arbitrarily large~$r$, since we can just forget about faces.
\end{proof}

We remark that while Theorem~\ref{thm:dim-3} is asymptotic with bounds for the size of the required point set that
are potentially far from optimal, Theorem~\ref{thm:balanced} gives that any balanced Tverberg admissible $r$-tuple
is Tverberg prescribable already for point sets of size at least $(r-1)(d+2)+1$, and this number is known to be optimal.

\begin{problem}
	Let $(d_1, \dots, d_r)$ be a Tverberg admissible $r$-tuple. Is it true that for any point set of size ${(r-1)(d+2)+1}$ there are
	pairwise disjoint subsets $X_1, \dots, X_r$ with $|X_i| = d_i+1$ such that $\conv X_1 \cap \dots \cap \conv X_r \ne \emptyset$?
	If not, what are good upper bounds for the required number of points?
\end{problem}

Note that the problem above is open even for $d_1 = \dots = d_r$ and $r$ with at least two distinct prime divisors.
P\'or's announced proof of Conjecture~\ref{conj} settles the affine AP Conjecture in full generality:

\begin{theorem}
\label{thm:conj-implies-prescr}
	Conjecture~\ref{conj} implies Conjecture~\ref{conj:affine-AP}.
\end{theorem}

\begin{proof}
	Given nonnegative integers $d_1, \dots, d_r$ with $\sum d_i = (r-1)d$
	such that each $d_i$ is at least $\lfloor \frac{d}{2} \rfloor$ there is a colorful $r$-partition of $\{1,2, \dots, {(r-1)(d+1)+1}\}$
	such that the $r$ parts have sizes~${d_i+1}$. This implies Conjecture~\ref{conj:affine-AP} given that this colorful partition
	must occur for sufficiently large point sets; this is guaranteed by Conjecture~\ref{conj}.
	In order to show the existence of this colorful $r$-partition first note that $\{1, \dots, d\}$ can be split into
	$r$ (possibly empty) parts $A_1, \dots, A_r$ with $|A_i| = d-d_i$ and no $A_i$ contains two consecutive integers.
	Simply assign the odd integers in order to $A_1$ first, then $A_2$, and so on; continue with the even integers in the same
	fashion once all odd ones are assigned to some~$A_i$. Only the set $A_1$ may have cardinality $\lceil\frac{d}{2}\rceil$, while
	$A_i$ for $i \ge 2$ has cardinality at most $\lfloor\frac{d}{2} \rfloor$. This ensures that no $A_i$ will receive two consecutive 
	integers by the procedure above.
	
	The required colorful partition of $\{1,2, \dots, {(r-1)(d+1)+1}\}$ can now be constructed as follows: if $k \in A_i$ place
	$(r-1)k+1$ in~$X_i$. Place one of the points $\{1, \dots, r-1\}$ in each $X_j$ with $r \notin X_j$.
	Place the $r-2$ points strictly between $(r-1)k+1$ and $(r-1)(k+1)+1$ each in one of the sets 
	$X_j$ with $(r-1)k+1 \notin X_j$ and $(r-1)(k+1)+1 \notin X_j$. The last $r-1$ points $\{(r-1)d+2, \dots, (r-1)(d+1)+1\}$
	can be put one each into the sets $X_j$ with $(r-1)d+1 \notin X_j$. By construction the resulting partition is colorful
	and $|X_j| = d-d_j$.
\end{proof}

If any sufficiently large point set in~$\R^d$ contains a Tverberg partition with one face of dimension~$k$, then we have seen
that necessarily $k \ge \lfloor \frac{d}{2} \rfloor$ by placing points in cyclic position. We will now show an even better lower
bound for the dimension in the continuous case:

\begin{theorem} 
\label{thm:prescribable}
	If $(d_1, \dots, d_r)$ is Tverberg prescribable for $r$ and $d$ then $d_i \ge \frac{r-1}{r}(d-1)$ for all~$i$.
\end{theorem}

\begin{proof}
	Let $d_1 \le d_2 \dots \le d_r$ be Tverberg prescribable for $r$ and~$d$. For some positive integer $N$ let 
	${g\colon \Delta_N \longrightarrow \R^{d-1}}$ be an affine strong general position map. 
	See Perles and Sigron~\cite{perles2014} for the notion of strong general position; if a point set in~$\R^d$ is in strong general
	position then the codimension of the intersection of convex hulls of any $r$ pairwise disjoint sets is the sum
	of codimensions  or~$d+1$, whichever is smaller. Here a convex set in $\R^d$ has codimension $d+1$ if and only if it is empty.
	Then the map
	$f\colon \Delta_N \longrightarrow \R^d, x \mapsto (g(x), \dist(x, \Delta_N^{(d_1)}))$ that measures the distance to the 
	$d_1$-skeleton~$\Delta_N^{(d_1)}$ in the last component (say for the standard Euclidean metric on the simplex) is continuous.
	Thus for sufficiently large $N$ we find pairwise disjoint faces $\sigma_1, \dots, \sigma_r$ of~$\Delta_N$ of dimensions
	$\dim\sigma_i = d_i$ and such that ${f(\sigma_1) \cap \dots \cap f(\sigma_r)} \ne \emptyset$.
	Choose points $x_i \in \sigma_i$ with $f(x_1) = f(x_2) = \dots = f(x_r)$. Then 
	$x_1 \in \sigma_1 \subseteq \Delta_N^{(d_1)}$, $g(x_1) = g(x_2) = \dots = g(x_r)$, and
	$0 = \dist(x_1, \Delta_N^{(d_1)}) = \dist(x_2, \Delta_N^{(d_1)}) = \dots = \dist(x_r, \Delta_N^{(d_1)})$.
	
	Thus the minimal supporting faces of the points $x_1, \dots, x_r$ determine $r$ pairwise disjoint faces 
	in $\Delta_N^{(d_1)}$ that have a common point of intersection under~$g$. By strong general position we have 
	the bound $r(d-1-d_1) \le d-1$ and thus $d_1 \ge \frac{r-1}{r}(d-1)$.
\end{proof}

For $d = 3$ and $r \ge 3$ the lower bound $\frac{r-1}{r}(d-1)$ is strictly larger than one, and thus the $r$-tuple
$(1,2,3,\dots, 3)$ is not Tverberg prescribable. This observation combined with Theorem~\ref{thm:3-dim-prescr}
immediately yields the following corollary giving an elementary example of the different nature of Tverberg-type
results for affine and continuous maps.

\begin{corollary}
\label{cor:prescr}
	In dimension $d = 3$ and for every $r \ge 3$ there is a Tverberg admissible $r$-tuple that is affinely Tverberg prescribable
	but not Tverberg prescribable. In particular, Conjecture~\ref{conj:AP} and the continuous relaxation of Conjecture~\ref{conj} 
	are wrong for any~$r \ge 3$.
\end{corollary}

The continuous maps $f\colon \Delta_N \longrightarrow \R^3$ we constructed that avoid Tverberg partitions involving 
faces of dimension one are tame; in fact, they are affine on the first barycentric subdivision of the simplex.

Let $(d_1, \dots, d_r)$ be Tverberg admissible for $d$ and~$r$. Recall that $(d_1, \dots, d_r)$ is \emph{balanced}
if $|d_i - d_j| \le 1$ for all~$i$ and~$j$. Joji\'c, Vre\'cica, and \v Zivaljevi\'c~\cite{jojic2017, jojic2015} showed recently in a series of two
technical papers that for $r$ a power of a prime every balanced Tverberg admissible $r$-tuple is Tverberg
prescribable. Here we want to give short and simple proof of this result inspired by the papers of Sarkaria~\cite{sarkaria1990, sarkaria1991}.

Let $\Sigma_{N+1,r}^{k_1, \dots, k_r} = \bigcup_{\pi \in \Sym_r} (\Delta_N^{(k_{\pi(1)}-1)} * \dots * \Delta_N^{(k_{\pi(r)}-1)})_\Delta$, that is,
$\sigma_1 * \dots *\sigma_r$ is a face of $\Sigma_{N+1,r}^{k_1, \dots, k_r}$ if and only if the $\sigma_i$ are pairwise disjoint faces
of~$\Delta_N$ such that $\dim \sigma_i \le k_{\pi(i)}-1$ for some permutation $\pi \in \Sym_r$ that is independent of~$i$.
These \emph{symmetric multiple chessboard complexes} were introduced and studied by Joji\'c, Vre\'cica, and \v Zivaljevi\'c,
since they are the appropriate configuration space for the AP conjecture. 
They are $\Sym_r$-invariant subcomplexes of~$(\Delta_N)^{*r}_\Delta$, the \emph{$r$-fold deleted join} of~$\Delta_N$, that
contains a join $\sigma_1 * \dots * \sigma_r$ of faces $\sigma_i$ of~$\Delta_N$ whenever the $\sigma_i$ are pairwise disjoint.
In general, $K^{*r}_\Delta$ denotes the \emph{$r$-fold deleted join} of the simplicial complex~$K$.
Denote by $W_r = \{(y_1, \dots, y_r) \in \R^r \: | \: \sum y_i = 0\}$ the standard representation of the symmetric group~$\Sym_r$.
We refer to Matou\v sek~\cite{matousek2008} for further details and notation.
The balanced case of the AP conjecture will easily follow from the lemma below that can be seen as lifting the ``constraint method''
of Blagojevi\'c, Ziegler, and the author~\cite{blagojevic2014} to the associated configuration space. This yields reasoning similar
to that employed by Sarkaria~\cite{sarkaria1990, sarkaria1991} in earlier papers.

\begin{lemma}
\label{lem:constraint}
	Let $d_1, \dots, d_r$ be balanced with $\sum d_i = (r-1)d$ and $N = (r-1)(d+2)$. 
	Then there is a $\Sym_r$-equivariant map $\Phi\colon(\Delta_N)^{*r}_\Delta \longrightarrow W_r$
	such that $\Phi^{-1}(0) = \Sigma_{N+1,r}^{d_1+1, \dots, d_r+1}$.
\end{lemma}

\begin{proof}
	We define the map $\Phi$ as an affine map on the barycentric subdivision~$((\Delta_N)^{*r}_\Delta)'$ of~$(\Delta_N)^{*r}_\Delta$.
	Label the vertices of~$\Delta_N$ by $1, 2, \dots, N+1$ in some arbitrary order. Define $\Phi$ to be zero on any vertex
	subdividing a face of~$K=\Sigma_{N+1,r}^{d_1+1, \dots, d_r+1}$. If $v$ subdivides face $\sigma_1 * \dots * \sigma_r$ not in~$K$ map $v$ to
	the standard basis vector $e_i \in \R^r$, where $\sigma_i$ has the vertex of smallest label among all faces $\sigma_j$ of the lowest dimension.
	The map $\Phi$ maps to~$W_r$ after projecting along the diagonal~$D = \{(x, \dots, x)\}$ and is $\Sym_r$-equivariant 
	and zero on~$K$ by definition. It is left to check that for $x \in (\Delta_N)^{*r}_\Delta \setminus K$ we have that $\Phi(x) \notin D$.
	Let $x \in (\Delta_N)^{*r}_\Delta$ with $\Phi(x) \in D$. Then either $\Phi(x) = 0$ and $x \in K$ or otherwise $x$ is contained in
	a face of $((\Delta_N)^{*r}_\Delta)'$ that has vertices $v_1, \dots, v_r$ with $\Phi(v_i) = e_i$. The $v_i$ subdivide faces
	$\sigma_1^{(i)} * \dots * \sigma_r^{(i)}$ of $(\Delta_N)^{*r}_\Delta$ that are totally ordered by inclusion in some way.
	W.l.o.g. $\sigma_1^{(i)} * \dots * \sigma_r^{(i)} \subsetneq \sigma_1^{(j)} * \dots * \sigma_r^{(j)}$ if $i < j$.
	
	Let $k = \lfloor \frac{r-1}{r}d \rfloor$ such that $d_1 = \dots = d_t = k$ and $d_{t+1} = \dots = d_r = k+1$.
	Record the dimensions in an $r$-by-$r$ matrix $M = (\dim \sigma_\ell^{(i)})_{1 \le \ell, i \le r}$, where rows correspond to
	constant~$i$ and columns correspond to constant~$\ell$. The integers in the columns are nondecreasing from top to 
	bottom. Moreover, in the $i$-th row the entry in column~$i$ is weakly the smallest among all entries of that row since
	$\Phi(v_i) = e_i$. Also, the entries in the $i$-th column, $1 \le i \le r-1$ must increase from row~$i$ to~${i+1}$ since
	$\Phi(v_{i+1}) \ne e_i$. As the faces $\sigma_1^{(i)} * \dots * \sigma_r^{(i)}$
	may at most involve $N+1 = (r-1)(d+2)+1$ vertices the sum of each row of~$M$ is bounded from above by~${(r-1)(d+1)}$.
	Lastly, as the faces $\sigma_1^{(i)} * \dots * \sigma_r^{(i)}$ are not contained in~$K$, each row either contains an
	entry that is at least $k+2$ or at least $r-t+1$ entries equal to~${k+1}$. 
	
	This leads to a contradiction; such a matrix does not exist. If the first row contains an entry that is at least~${k+2}$, then 
	the last row has all entries $\ge k+2$. If the first row contains $r-t+1$ entires equal to $k+1$, then the last row has all
	entries $\ge k+1$ and $r-t$ entries $\ge k+2$. In either case the sum of the last row is at least $\sum d_i +1 = (r-1)(d+1)+1$,
	which is a contradiction.
\end{proof}

As an example consider why $\Sigma_{15,3}^{4,4,5}$ is the zero set of an equivariant
map $(\Delta_{14})^{*3}_\Delta \longrightarrow W_3$. A face $\sigma_1 * \sigma_2 * \sigma_3$ that is not contained in $\Sigma_{15,3}^{4,4,5}$ either
contains $\sigma_i$ of dimension at least five or two $\sigma_i$ of dimension four. Now the proof above provides us with three
distinct such faces that are totally ordered by inclusion. Moreover, each $\sigma_i$ has to be the face of (weakly) least dimension 
once. If the smallest face has $\sigma_i$ of respective dimensions $(0,4,4)$ then the largest of the three faces in the total order
has at least dimensions $(4,4,5)$. Such a face involves $16$ vertices, but by definition of $\Sigma_{15,3}^{4,4,5}$ the face may involve 
at most $15$ vertices.

\begin{lemma}[\"Ozaydin~\cite{ozaydin1987}]
\label{lem:no-equiv-map}
	Let $r$ be a power of a prime, $d\ge 1$ an integer, and $N\ge (r-1)(d+2)$. Then there is no $\Sym_r$-equivariant map
	$(\Delta_N)^{*r}_\Delta \longrightarrow S(W_r^{\oplus (d+2)})$.
\end{lemma}

We can use Lemma~\ref{lem:no-equiv-map} as a ``blackbox result'' and the map whose existence is guaranteed by 
Lemma~\ref{lem:constraint} acts as a constraint function. The balanced case of the AP conjecture easily follows. 
In addition to providing a significantly simplified approach to this result, our proof shows that the ``indirect methods'',
compare Joji\'c, Vre\'cica, and \v Zivaljevi\'c~\cite[Section 1.2]{jojic2015}, are sufficiently strong to prove the balanced case 
of the AP conjecture, which settles a problem raised in~\cite{jojic2015}. Very general Tverberg-type results that can be 
deduced from indirect methods can be found in~\cite{frick2016}.

\begin{theorem}[Joji\'c, Vre\'cica, and \v Zivaljevi\'c~\cite{jojic2015}]
\label{thm:balanced}
	Let $r$ be a prime power and $d_1, \dots, d_r$ be balanced with $\sum d_i = (r-1)d$. Then $(d_1, \dots, d_r)$ is Tverberg prescribable;
	every continuous map $\Delta_{(r-1)(d+2)} \longrightarrow \R^d$ identifies points from $r$ pairwise disjoint faces of respective dimensions
	$d_1, \dots, d_r$.
\end{theorem}

\begin{proof}
	Let $N = (r-1)(d+2)$ and let $f\colon \Delta_N \longrightarrow \R^d$ be continuous. Suppose that for any $r$ pairwise disjoint 
	faces $\sigma_1, \dots, \sigma_r$ of~$\Delta_N$ with dimensions $\dim \sigma_i = d_i$ we have that 
	$f(\sigma_1) \cap \dots \cap f(\sigma_r) = \emptyset$. Then the $\Sym_r$-equivariant map 
	$F\colon (\Delta_N)^{*r}_\Delta \longrightarrow (\R^{d+1})^r, \lambda_1x_1+ \dots + \lambda_rx_r \mapsto (\lambda_1, \lambda_1f(x_1), \dots, \lambda_r, \lambda_rf(x_r))$
	avoids the diagonal $D = \{(y_1,\dots, y_r) \in (\R^{d+1})^r \: | \: y_1 = \dots = y_r\}$ on~$\Sigma_{N+1,r}^{d_1+1, \dots, d_r+1}$. 
	This is because a point $ \lambda_1x_1+ \dots + \lambda_rx_r \in \Sigma_{N+1,r}^{d_1+1, \dots, d_r+1}$ with $F(\lambda_1x_1+ \dots + \lambda_rx_r) \in D$
	would imply $\lambda_1 = \lambda_2 = \dots = \lambda_r$ and thus $f(x_1) = f(x_2) = \dots = f(x_r)$, where by definition of the complex
	$\Sigma_{N+1,r}^{d_1+1, \dots, d_r+1}$ the $x_i$ are contained in pairwise disjoint faces of dimensions $d_1, \dots, d_r$. However, $f$ does not map
	$r$ points from pairwise disjoint faces of those dimensions to the same point. Thus by orthogonal
	projection to $D^\perp$ the map~$F$ induces an $\Sym_r$-map $\overline{F}\colon (\Delta_N)^{*r}_\Delta \longrightarrow W_r^{\oplus (d+1)}$ avoiding the
	origin on~$\Sigma_{N+1,r}^{d_1+1, \dots, d_r+1}$. Using the map $\Phi$ in Lemma~\ref{lem:constraint} we obtain an $\Sym_r$-map
	$(\Delta_N)^{*r}_\Delta \longrightarrow W_r^{\oplus(d+2)}, x \mapsto (\overline{F}(x), \Phi(x))$ that avoids the origin. Radially projecting to the unit sphere
	yields a contradiction to Lemma~\ref{lem:no-equiv-map}.
\end{proof}

For every $r \ge 2$ there is at least one Tverberg admissible $r$-tuple that is Tverberg prescribable:

\begin{corollary}
	Let $r \ge d \ge 2$ be integers. Then the $r$-tuple $(d-1, \dots, d-1, d, \dots, d)$ that contains $d$ times the entry $d-1$ is Tverberg prescribable.
\end{corollary}

\begin{proof}
	Let $q \ge r$ be a prime power. By Theorem~\ref{thm:balanced} the $q$-tuple $(d-1, \dots, d-1, d, \dots, d)$ that contains $d$ 
	times the entry $d-1$ is Tverberg prescribable. We can disregard the last $q-r$ faces.
\end{proof}

We remark that for $r$ not a prime power even constant $r$-tuples need not be Tverberg prescribable; see~\cite[Theorem 4.2]{blagojevic2015}.

\section{Connectivity bounds for symmetric multiple chessboard complexes}

Here we show how our geometric constructions of the previous section provide connectivity bounds from above for certain natural symmetric 
simplicial complexes --- symmetric multiple chessboard complexes. 
For positive integers $n$ and $m$ the simplicial complexes
$\Delta(m,n) = [m]^{*n}_\Delta = \Delta(n,m)$, that is, the $n$-fold deleted join of a set of $m$ discrete points, 
are called \emph{chessboard complexes}; $\Delta(m,n)$ is the matching complex of the complete
graph~$K_{m,n}$, that is, the vertices are in bijective correspondence with the edges of~$K_{m,n}$ and faces consist of sets of edges that do
not share vertices. Another description identifies the faces with the positions on an $m$-by-$n$ chessboard, and faces consist of all non-taking rook
placements on the chessboard. These complexes naturally appear in the investigation of several algebraic and combinatorial problems
and thus their combinatorial and topological properties have been of major interest; see Jonsson~\cite{jonsson2008} and Wachs~\cite{wachs2003}
for an overview of the vast literature. Here we mention in particular lower bounds for the connectivity of 
chessboard complexes due to Bj\"orner, Lov\'asz, Vre\'cica, and \v Zivaljevi\'c~\cite{bjorner1994}, upper bounds for their connectivity
due to Shareshian and Wachs~\cite{shareshian2007}, and their shellability for those parameters where they are maximally connected
due to Ziegler~\cite{ziegler1994}.

Chessboard complexes are the appropriate configuration spaces for several Tverberg-type problems; see in particular Vre\'cica and 
\v Zivaljevi\'c~\cite{zivaljevic1992, vrecica1994, vrecica2011} and Blagojevi\'c, Matschke, and Ziegler~\cite{blagojevic2009}.
A natural generalization of chessboard complexes that arises naturally for Tverberg-type results was studied by Joji\'c, Vre\'cica, and \v Zivaljevi\'c~\cite{jojic2015}:
the \emph{multiple chessboard complex} $\Delta_{m,n}^{k_1, \dots, k_n}$ has as
vertex set the squares of an $m$-by-$n$ chessboard and a face for any set of rooks with at most one rook per row and at most
$k_i$ rooks in column~$i$. Another way to represent the complex $\Delta_{m,n}^{k_1, \dots, k_n}$ is as the deleted join
of simplex skeleta $(\Delta_{m-1}^{(k_1-1)} * \dots * \Delta_{m-1}^{(k_n-1)})_\Delta$, that is, the join $\sigma_1 * \dots * \sigma_n$ is
a face of $(\Delta_{m-1}^{(k_1-1)} * \dots * \Delta_{m-1}^{(k_n-1)})_\Delta$ if the $\sigma_i$ are pairwise disjoint faces of~$\Delta_{m-1}$
such that $\dim \sigma_i \le k_i-1$. To establish that these complexes have high connectivity, they showed:

\begin{theorem}[Joji\'c, Vre\'cica, and \v Zivaljevi\'c~\cite{jojic2017}]
	For $m \ge k_1+\dots+k_n+n-1$ the complex $\Delta_{m,n}^{k_1, \dots, k_n}$ is shellable.
\end{theorem}

In a subsequent paper Joji\'c, Vre\'cica, and \v Zivaljevi\'c~\cite{jojic2015} had the insight that the balanced case of the AP conjecture for prime powers~$r$
can be resolved by extending (for specific instances) the shelling order to a symmetrized version of~$\Delta_{m,n}^{k_1, \dots, k_n}$: recall that the 
\emph{symmetric multiple chessboard complex} $\Sigma_{m,n}^{k_1, \dots, k_n}$ contains a face $\sigma_1 * \dots * \sigma_n$
if the $\sigma_i$ are pairwise disjoint faces of~$\Delta_{m-1}$ such that there exists a permutation $\pi \in \Sym_n$ with $\dim \sigma_{\pi(i)} \le k_i-1$.

\begin{theorem}[Joji\'c, Vre\'cica, and \v Zivaljevi\'c~\cite{jojic2015}]
	Let $(k_1, \dots, k_n)$ be a balanced sequence of nonnegative integers and $m \ge k_1+\dots+k_n+n-1$. Then 
	$\Sigma_{m,n}^{k_1, \dots, k_n}$ is shellable.
\end{theorem}

Joji\'c, Vre\'cica, and \v Zivaljevi\'c do not address the question whether it is necessary for the shellability 
of $\Sigma_{m,n}^{k_1, \dots, k_n}$ that $(k_1, \dots, k_n)$ is balanced. Here we show that
if a sequence is sufficiently far from balanced, shellability fails. We establish this non-shellability result as
a consequence of the continuous maps constructed in Theorem~\ref{thm:prescribable}. In particular, 
Tverberg-type intersection results and intersection patterns of convex hulls in Euclidean space may be
used to find obstructions for the shellability (and upper bounds for the connectivity) of certain symmetric 
simplicial complexes. We will need the following lemma.

\begin{lemma}[Volovikov~\cite{volovikov1996-2}]
\label{lem:volovikov}
	Let $p$ be a prime and $G = (\Z/p)^n$ an elementary abelian $p$-group. Suppose that $X$ and $Y$ are fixed-point free 
	$G$-spaces such that $\widetilde{H}(X;\Z/p) \cong 0$ for all $i \le n$ and $Y$ is an $n$-dimensional cohomology sphere 
	over~$\Z/p$. Then there does not exist a $G$-equivariant map $X \longrightarrow Y$.
\end{lemma}

\begin{theorem}
\label{thm:nonshellable}
	Let $d_1, \dots, d_r$ be nonnegative integers with $\sum d_i = (r-1)d$ and suppose that there exists a $j$ with $d_j < \frac{r-1}{r}(d-1)$.
	Let $m=(r-1)(d+1)-1$.
	Then for any $n$ and $r$ a power of a prime $\Sigma_{n,r}^{d_1+1, \dots, d_r+1}$ is not $m$-connected.
	In particular, for $n \ge m+2$ the complex $\Sigma_{n,r}^{d_1+1, \dots, d_r+1}$ is $(m+1)$-dimensional and thus not shellable. 
\end{theorem}

\begin{proof}
	By Theorem~\ref{thm:prescribable} there is a map $f\colon \Delta_{n-1} \longrightarrow \R^d$ such that for any $r$
	pairwise disjoint faces $\sigma_1, \dots, \sigma_r$ of~$\Delta_{n-1}$ with dimensions $\dim \sigma_i = d_i$ the 
	intersection $f(\sigma_1) \cap \dots \cap f(\sigma_r)$ is empty. Then, as in the proof of Theorem~\ref{thm:balanced}, 
	the $\Sym_r$-equivariant map 
	$(\Delta_{n-1})^{*r}_\Delta \longrightarrow (\R^{d+1})^r, \lambda_1x_1+ \dots + \lambda_rx_r \mapsto (\lambda_1, \lambda_1f(x_1), \dots, \lambda_r, \lambda_rf(x_r))$
	avoids the diagonal $D = \{(y_1,\dots, y_r) \in (\R^{d+1})^r \: | \: y_1 = \dots = y_r\}$ on the subcomplex~$\Sigma_{n,r}^{d_1+1, \dots, d_r+1}$. Thus by orthogonal
	projection to $D^\perp$ and radial projection to the unit sphere this induces an $\Sym_r$-map 
	$F\colon \Sigma_{n,r}^{d_1+1, \dots, d_r+1} \longrightarrow S(W_r^{\oplus (d+1)})$. Let $p$ be a prime with $r =p^n$, then 
	$(\Z/p)^n$ embeds into $\Sym_r$ in a natural way (as explained in~\cite{ozaydin1987}) such that the action of $(\Z/p)^n$ has no fixed
	points on~$S(W_r^{\oplus (d+1)})$. So by Lemma~\ref{lem:volovikov} the symmetric multiple chessboard complex
	$\Sigma_{n,r}^{d_1+1, \dots, d_r+1}$ is not $m$-connected, and thus not shellable.
\end{proof}

This shows in particular that Theorem~\ref{thm:3-dim-prescr} is an affine Tverberg-type result that does not follow from
the usual topological approach of showing the nonexistence of an equivariant map from the associated configuration space
to a representation sphere of the symmetric group.

\section{Bourgin--Yang orbit collapsing: affine vs. continuous}

We present a simple construction that exhibits a Tverberg-type result that holds in the affine setting for any intersection multiplicity
and has a continuous relaxation if and only if the intersection multiplicity is a prime. The affine version is an extension of a result
of Sober\'on~\cite{soberon2015}. To state Sober\'on's result we need the following definition: a point $x \in [r]^{*n}$ can be written as 
$x = \lambda_1x_1 + \dots + \lambda_nx_n$ with $x_i \in [r]$, $\lambda_i \ge 0$, and $\sum \lambda_i = 1$; given a second point
$y = \lambda_1y_1 + \dots + \lambda_ny_n$ with the same coefficients but perhaps different $y_i \in [r]$, we say that $x$ and $y$
have \emph{equal barycentric coordinates}.

\begin{theorem}[Sober\'on~\cite{soberon2015}]
\label{thm:barycentric}
	Let $n \ge (r-1)d$ and $f\colon [r]^{*(n+1)} \longrightarrow \R^d$ be affine. Then there are $x_1, \dots, x_r \in [r]^{*(n+1)}$
	with equal barycentric coordinates in pairwise disjoint faces of $[r]^{*(n+1)}$ such that $f(x_1) = \dots = f(x_r)$.
\end{theorem}

We have the following isomorphisms of simplicial complexes $[r]^{*(n+1)} \cong (\pt^{*r}_\Delta)^{*(n+1)} \cong (\pt^{*(n+1)})^{*r}_\Delta
\cong (\Delta_n)^{*r}_\Delta$ and thus $[r]^{*(n+1)}$ is the configuration space for Tverberg's theorem.
Sober\'on asked whether there is a topological relaxation of his result. Such a continuous generalization for $r$ a power of
a prime was proven in~\cite{blagojevic2014}. Moreover the proof method in the same way yields a new proof of Theorem~\ref{thm:barycentric}
by reducing it to Tverberg's theorem. Here we first extend Theorem~\ref{thm:barycentric} and then show that this extension
has a topological generalization if and only if the intersection multiplicity $r$ is a prime. We interpret Theorem~\ref{thm:barycentric} as
an approximation of an affine version of Bourgin--Yang orbit collapsing results; for rather general such theorems see Fadell and
Husseini~\cite{fadell1988}. In particular, our topological generalization, Theorem~\ref{thm:top-orbit}, also readily follows from the methods developed there.
Our affine version is an orbit collapsing
result that can be deduced from Sarkaria's linear Borsuk--Ulam theorem~\cite{sarkaria2000}, which itself is a corollary of B\'ar\'any's
colorful Carath\'eodory theorem~\cite{barany1982}. Instead, for ease of exposition, we will directly deduce the affine version from the colorful
Carath\'eodory theorem, which states that if $X_0, \dots, X_d \subseteq \R^d$ satisfy $0 \in \bigcap_i \conv X_i$ then there are
$x_0 \in X_0, \dots, x_d \in X_d$ such that $0 \in \conv\{x_0, \dots, x_d\}$. The group~$\Z/r$ naturally acts on~$[r]$ by shifts,
that is, we can think of~$[r]$ as the group~$\Z/r$ which naturally acts on itself. So $\Z/r$ acts diagonally on the vertex set of
$[r]^{*(n+1)}$ and thus on the complex $[r]^{*(n+1)}$ itself by extending affinely. If $x$ and $y$ are two points from the same 
$\Z/r$-orbit in~$[r]^{*(n+1)}$ then they have equal barycentric coordinates. Moreover, two distinct points from the same $\Z/r$-orbit
are necessarily in disjoint faces. Our proof of the following theorem answers a question of Sober\'on~\cite[Section~4]{soberon2015},
whether there is an orbit-collapsing proof of his result.

\begin{theorem}
\label{thm:aff-orbit}
	Let $r \ge 2$ and $d \ge 1$ be integers. For any integer $n \ge (r-1)d$ 
	and any affine map $f\colon [r]^{*(n+1)} \longrightarrow \R^d$
	there is a $\Z/r$-orbit in $[r]^{*(n+1)}$ that $f$ collapses to one 
	point.
\end{theorem}

\begin{proof}
	Let $t \in \Z/r$ be a generator. Define $F\colon [r]^{*(n+1)} \longrightarrow (\R^d)^r$ by
	$F(x) = {(f(x), f(t\cdot x), \dots, f(t^{r-1}\cdot x)})$. We need to show that the image of~$F$
	intersects the diagonal $D = \{(y_1, \dots, y_r) \in (\R^d)^r \: | \: {y_1 = y_2 = \dots = y_r}\}$.
	Let $\pi\colon (\R^d)^r \longrightarrow D^\perp$ be the orthogonal projection. For any vertex
	$v$ of $[r]^{*(n+1)}$ the barycenter of the set $\{F(v), F(t\cdot v), \dots, F(t^{r-1}\cdot v)\}$ is
	contained in the diagonal~$D$ and thus $$0 \in {\conv\{\pi(F(v)), \pi(F(t\cdot v)), \dots, \pi(F(t^{r-1}\cdot v))\}}.$$
	for each vertex~$v$. There are $n+1 \ge (r-1)d+1$ of these vertex orbits and $D^\perp$ has 
	dimension~$(r-1)d$. Thus by the colorful Carath\'eodory theorem there is a choice of one 
	vertex per orbit such that the convex hull of these vertices captures the origin. But a choice of
	one vertex per orbit determines a face of~$[r]^{*(n+1)}$, so this proves our claim.
\end{proof}

This extension of Theorem~\ref{thm:barycentric} has a topological generalization if and only if $r$ is a prime.
The proof that Theorem~\ref{thm:aff-orbit} holds if $f$ is merely continuous and $r$ is a prime is essentially
the same; now using a theorem of Dold instead of the colorful Carath\'eodory theorem:

\begin{lemma}[Dold~\cite{dold1983}]
\label{lem:dold}
	Let a non-trivial finite group $G$ act on an $n$-connected CW-complex~$K$ and act linearly on an 
	$(n+1)$-dimensional real vector space~$V$. Suppose that the action of $G$ on $V \setminus \{0\}$ is free.
	Then any $G$-equivariant map $K \longrightarrow V$ has a zero.
\end{lemma}

\begin{theorem}
\label{thm:top-orbit}
	Let $r$ be a prime and $d \ge 1$ an integer. For any integer $n \ge (r-1)d$ 
	and any continuous map $f\colon [r]^{*(n+1)} \longrightarrow \R^d$
	there is a $\Z/r$-orbit in $[r]^{*(n+1)}$ that $f$ collapses to one 
	point.
\end{theorem}

\begin{proof}
	The complex $[r]^{*(n+1)}$ is $(n-1)$-connected. Repeat the proof of Theorem~\ref{thm:aff-orbit}.
	The map $\pi\circ F$ is $\Z/r$-equivariant and for $r$ a prime the action on $D^\perp$ is free away from
	the origin. Thus an application of Lemma~\ref{lem:dold} instead of the colorful Carath\'eodory theorem
	finishes the proof.
\end{proof}

In fact, in a similar way one can prove strengthenings of Theorem~\ref{thm:top-orbit}: let $C_{2r}$ be a circle
on $2r$ vertices, and let $t$ be a generator of~$\Z/r$. Define a $\Z/r$-action on $C_{2r}$ by $t$ rotating the
circle by two vertices, so that there are two disjoint orbits of vertices. Then $C_{2r}$ equivariantly embeds 
into~$[r]^{*2}$. Now suppose that $n \ge (r-1)d$ for some prime~$r$ and that $n+1$ is even, say ${n+1=2k}$.
Then $C_{2r}^{*k}$ is a subcomplex of $[r]^{*(n+1)}$, which is a proper subcomplex for $r\ge 3$. Now since
$C_{2r}^{*k}$ is homeomorphic to $S^{2k-1}$ it is $(n-1)$-connected, and thus, as before, for any continuous 
map $f\colon C_{2r}^{*k} \longrightarrow \R^d$ there is a $\Z/r$-orbit in $C_{2r}^{*k}$ that $f$ collapses to one 
point. The affine analog of this result for multiplicities $r$ with at least two distinct prime divisors is an open
problem.

That Theorem~\ref{thm:top-orbit} does not hold for composite numbers~$r$ is a simple consequence of 
the existence of certain equivariant maps --- first constructed by \"Ozaydin~\cite{ozaydin1987}. He showed that for composite~$r$
and $n \le (r-1)d$ there exists a $\Z/r$-equivariant map $(\Delta_n)^{*r}_\Delta \cong  [r]^{*(n+1)} \longrightarrow S(W_r^{\oplus d})$.
This was recently extended to arbitrary $n$ by Basu and Ghosh~\cite{basu2015} provided that $r$ is not a power of a prime.

\begin{theorem}
\label{thm:top-orbit-fails}
	Let $r \ge 4$ be a composite number and let $d \ge 1$ be an integer. 
	For $n = (r-1)d$ there exists a continuous map 
	$f\colon [r]^{*(n+1)} \longrightarrow \R^d$ that does not
	collapse any $\Z/r$-orbit in $[r]^{*(n+1)}$ to one point.
\end{theorem}

\begin{proof}
	The representation $W_r^{\oplus d}$ embeds into $(\R^d)^r$ as the orthogonal complement $D^\perp$
	of the diagonal~$D$. Thus \"Ozaydin's construction yields a $\Z/r$-equivariant map 
	$[r]^{*(n+1)} \longrightarrow (\R^d)^r$ that avoids the diagonal. The projection of this map onto the 
	first $\R^d$-factor is the map~$f$.
\end{proof}

If we require $r$ to have at least two distinct prime divisors then the result of Basu and Ghosh~\cite{basu2015}
shows the existence of an $f\colon [r]^{*(n+1)} \longrightarrow \R^d$ as in the theorem above for any~$n$.
In fact, if $r$ is not a prime, and $c$ is a constant then for $d$ sufficiently large there is
a $\Z/r$-map (and even $\Sym_r$-map) $[r]^{*(n+1)} \longrightarrow S(W_r^{\oplus d})$ for $n = (r-1)d + c$, see~\cite[Section~5]{blagojevic2015}.
Thus the theorem above also holds in those situations.


\end{document}